\documentclass[11pt,a4paper]{article}
\usepackage{latexsym}
\usepackage{amssymb}
\usepackage{amsthm}
\usepackage{xcolor}
\usepackage{multicol}
\usepackage{amsmath}
\title{Sharp uniform bound for the quaternionic Monge-Amp\`ere equation on hyperhermitian manifolds}
\author{Marcin Sroka}
\date{}

\usepackage{geometry}
\newtheorem{theorem}{Theorem}[section]
\newtheorem{proposition}[theorem]{Proposition}
\newtheorem{lemma}[theorem]{Lemma}
\newtheorem{remark}[theorem]{Remark}

\newcommand{\hh}{\mathbb{H}}
\newcommand{\rr}{\mathbb{R}}

\newcommand{\ii}{\mathfrak{i}}
\newcommand{\jj}{\mathfrak{j}}
\newcommand{\kk}{\mathfrak{k}}

\newcommand{\pa}{\partial}
\newcommand{\bpar}{\overline{\partial}}
\newcommand{\jpar}{\partial_J}
\newcommand{\bjpar}{\overline{\partial_J}}

\newcommand{\ball}{B(z_0,2r)}
\begin{document}
\maketitle

\textbf{Abstract:} We provide the sharp $C^0$ estimate for the quaternionic Monge-Amp\`ere equation on any hyperhermitian manifold. This improves previously known results concerning this estimate in two directions. Namely, it turns out that the estimate depends only on $L^p$ norm of the right hand side for any $p>2$ (as suggested by the local case studied in \cite{Sr18}). Moreover, the estimate still holds true for any hyperhermitian initial metric - regardless of it being HKT as in the original conjecture of Alesker-Verbitsky \cite{AV10} - as speculated by the author in \cite{Sr21}. For completeness, we actually provide a sharp uniform estimate for many quaternionic PDEs, in particular those given by the operator dominating the quaternionic Monge-Amp\`ere operator, by applying the recent method of Guo and Phong \cite{GP22}.    



\section{Introduction}

In \cite{AV10} Alesker and Verbitsky posed an analogue of the famous Calabi conjecture, solved by Yau in \cite{Y78}, on certain hypercomplex manifolds. It asks whether, on hypercomplex manifold admitting HKT metric, every section of the canonical bundle, i.e. complex volume form, is associated to some HKT metric. If true it implies, in particular, that every representative of the first Bott-Chern class of the underlying complex manifold is Chern-Ricci curvature of an HKT metric and thus fits into the recently active area of studying Calabi-Yau type theorems for different classes of hermitian non-K\"ahler metrics. A much more elaborate discussion of this conjecture can be found in \cite{DS21}.

Above described problem reduces to solving certain second order fully nonlinear PDE, cf. (\ref{qma}) below, which Alesker studied earlier in local setting and called - the quaternionic Monge-Amp\`ere equation. This equation was also independently treated but again in local setting, i.e. the Dirichlet problem in a domain, by Harvey and Lawson in \cite{HL09}. 

In \cite{AV10} the conjecture was posed in a very conservative manner. Namely, under the assumptions that the hypercomplex manifold admits an HKT metric and the canonical bundle of the underlying complex manifold is holomorphically trivial. This corresponds to the setting in the original Calabi conjecture where the K\"ahler manifold satisfies $c_1=0$. If this conjecture is true the hypercomplex manifolds satisfying the mentioned assumptions would admit a very special HKT metric which can be thought of as hyperhermitian (torsion) analogue of Calabi-Yau metric (in place of HyperK\"ahler metrics which are very rigid and mysterious abjects). From the analytic point of view it is natural to expect that equation (\ref{qma}) is solvable more generally on every HKT manifold without any additional conditions. This corresponds to the full version of the original conjecture of Calabi - with the class of K\"ahler metrics exchanged for HKT ones, cf. \cite{AS17}. Even more generally, by taking into an account the result for the complex Monge-Amp\`ere equation on hermitian manifolds from \cite{TW10}, one can expect (\ref{qma}) to be solvable on any hyperhermitian or even almost hyperhermitian manifold - cf. the speculations in \cite{Sr21, DS21}. 

Currently, even the original conjecture from \cite{AV10} remains unsolved. The issue are a priori estimates for the equation (\ref{qma}). For the detailed discussion we refer to \cite{DS21}. We limit ourself to vaguely summarize the problems with obtaining the higher order estimates. These come from the fact that the hypercomplex structure is locally not flat i.e. not diffeomorphic to a domain in the flat quaternion space. This results in, when performing calculations for the application of the maximum principle, emergence of terms related to the curvature of the metric and the curvature of the hypercomplex structure. It turned out to be a hard task to control both simultaneously. However, when they do coincide, i.e. the initial metric is HyperK\"ahler, the higher order bounds were obtained in \cite{DS21}. This remains to be the most general sitting in which equation (\ref{qma}) was solved. 

Surprisingly, the $C^0$ estimate is known for the equation (\ref{qma}) on any HKT manifold. It was first shown by Alesker and Shelukhin \cite{AS17} who adopted the idea of Błocki \cite{B11}. Later, very lengthy proof from \cite{AS17} was much simplified in \cite{Sr19}. The method in the latter was strongly inspired by the sharpened version of the Moser iteration for nonlinear PDEs from \cite{TW10}. While the dependence of the estimate from \cite{AS17} on the right hand side is, stated there to be, on its $L^\infty$ norm, it seems to the author that once tracing the argument closely it may be reduced to $L^4$ norm. The dependence in \cite{Sr19} is on the $L^p$ norm where $p$ depends on the dimension of the manifold. One can wonder, what is the sharpest possible dependence of the $C^0$ estimate on the right hand side in (\ref{qma})? 

Let us consider, for a moment, the local setting. For the real Monge-Amp\`ere equation the uniform bound depends solely on $L^1$ norm of the right hand side, which is just Alexandrov maximum principle. For the complex Monge-Amp\`ere equation it is a deep result of Kołodziej, cf. \cite{K96}, that it depends on $L^p$ norm for any $p>1$ and one can not go to $p=1$. Later, Kołodziej proved, cf. \cite{K98}, that this local results has the global counterpart - the $C^0$ bound for the complex Monge-Amp\`ere equation on a compact K\"ahler manifold also depends only on $L^p$ norm of the right hand side for any $p>1$. Even more precisely the dependence in both situations in on some Orlicz norm of the right hand side. In \cite{Sr18} the author proved that for the Dirichlet problem for (\ref{qma}) in flat quaternionic space the uniform bound depends only on $L^p$ norm of the right hand side for any $p>2$ and that once can not go to $p=2$. Thus, this provides a sharp estimate in a local and flat setting. Taking that into consideration one can expect the uniform bound for the equation (\ref{qma}) to depend only on $L^p$ norm of the right hand side for $p>2$.

The main goal of the present paper is to derive such a sharp uniform estimates for the
quaternionic Monge-Amp\`ere and certain other quaternionic equations. For this, we
apply the method of Guo and Phong \cite{GP22} for fully non-linear equations satisfying a
specific structural condition on Hermitian manifolds. In our case, the role of this
structural condition is played by a new inequality between complex and quaternionic
Monge-Amp\`ere operators, which can be viewed as the quaternionic version of the classic
inequality of Cheng and Yau, cf. \cite{Be93}, between real and complex Monge-Amp\`ere operators. In this
context, we note that recognizing when a fully non-linear equation satisfies the
structural condition needed for the method of Guo and Phong is of considerable
interest, and that our inequality between mentioned operators
should be a useful addition to the conditions found by Harvey and Lawson \cite{HL22}. 

Let us also note that since it is still unknown whether equation (\ref{qma}) can be solved in smooth category, in general, obtaining the sharp uniform bound may be important for studying regularity of, possibly only weak, solutions to (\ref{qma}) (but our believe is that the conjecture from \cite{AV10} is true). The argument we present works not only for hypercomplex manifolds admitting HKT metrics but for any hyperhermitian metric. This provides partial confirmation, the $C^0$ bound, for the Conjecture 7.2.4. from \cite{Sr21}. We point out that regardless of sharpness of the estimate, no uniform bound was known for (\ref{qma}) on general hyperhermitian manifold before. It may be possible to adopt the reasoning in \cite{Sr19} from HKT to hyperhermitian setting. Definitely more torsion terms will appear in formula $(2.1)$ of \cite{Sr19} in that setting (as the form $\alpha$ will not be $\pa$ closed) but, as described above, surely this would not produce the sharp bound (if any).

Summarizing the last two paragraphs, the main theorem we prove is as follows.  

\begin{theorem} \label{mt}
Let $(M,I,J,K,g)$ be a closed, connected hyperhermitian manifold and let $\phi$ be a smooth solution of the quaternionic Monge-Amp\`ere equation 
\[\tag{1.1} \label{qma} \begin{cases} (\Omega + \pa \jpar \phi)^n = e^F \Omega^n \\ \Omega + \pa \jpar \phi \geq 0 \\ sup_M \phi = 0 \end{cases}. \] 
For any $p>2n$, where $n$ is the quaternionic dimension of $M$, there is a uniform bound
\[\label{qmab} \tag{1.2} - \inf_M \phi < C \] 
where the constant $C$ depends only on $p$, geometry of $(M,I,J,K,g)$ and $\parallel e^{2F} \parallel_{L^1(\log L)^p}$. Here \[\parallel e^{G} \parallel_{L^1(\log L)^p} = \int_M |G|^pe^{G} d vol_g\] 
for any smooth function $G$. 

In particular the bound (\ref{qmab}) depends only on $L^q$ norm of $e^F$ for any $q>2$.
\end{theorem}

The unified perspective outlined in \cite{HL09} suggests, at least in the PDE regime, to consider also more general equations of the eigenvalues of the quaternionic Hessian then just the quaternionic Monge-Amp\`ere one. Indeed, the Dirichlet problems for the equations factoring through eigenvalues of real, complex and quaternionic Hessians were solved in the viscosity sens in \cite{HL09}. The Dirichlet problem for such equations in domains of manifolds, supporting required geometric structures, was in turn  studied in \cite{HL11}.  One can view the real equations treated in \cite{L90, U02}, the complex ones treated in \cite{Sz18} and the quaternionic ones (\ref{qpde}) as the global, i.e. posed on closed manifold, versions of those. Equations (\ref{qpde}) were treated and solved under very restrictive assumptions in \cite{GZ22}. Provided the equation satisfies (\ref{3}) we much improve the uniform estimates proved there. Using a recent observation of \cite{HL22} we know this is the case in particular for any PDE defined by the operator dominating the quaternionic Monge-Amp\`ere equation. Theorem \ref{mt} becomes the special case of Proposition \ref{qpdebp}. We remark that the author is not aware of any geometric significance of those more general quaternionic Hessian equations.

Let us now describe the method of the proof of Proposition \ref{qpdebp} and its new element. We follow very closely the first step of the proof in \cite{GP22}, this time for the linearization of the quaternionic PDEs considered in (\ref{qpde}). During that step we add one more technique - the analogy of the idea of Cheng and Yau - the comparison between complex and quaternionic Monge-Amp\`ere operators. This seems to be unusual but allows us to use in this step, as an auxiliary equation, still \textit{the complex Monge-Amp\`ere equation} as in \cite{GP22} (even though we are dealing with PDEs not factoring through the eigenvalues of the complex Hessian). Because of that we are able to obtain the inequality required in Lemma 2 in \cite{GP22}. As noted there, this lemma does not depend in any way on the original equation and relies solely on nice properties - local exponential integrability - of plurisubharmonic functions appearing in the obtained inequality. 

We remark that the idea of using an auxiliary Monge-Amp\`ere equation has recently turned out to be extremely powerful for obtaining uniform estimates, and not only those, cf. \cite{B11,CC21,WWZ20,WWZ21,GPT21,GP22} to just name few recent and related to uniform estimates. An overwhelming and precise list of settings where it can be applied is presented in a recent survey \cite{GP23}. One can view our result as another example of such a setting, in which previously known methods either do not work or require considerably more effort and do not provide the sharp estimate.

In the next section we recall necessary background on hyperhermitian manifolds. In section 3 we discuss the inequality between determinants of a hermitian matrix and its hyperhermitian part as well as the resulting comparison between complex and quaternionic Monge-Amp\`ere operators on hypercomplex manifolds. In Section 4 we prove Proposition \ref{qpdebp} and consequently Theorem \ref{mt}.

\textbf{Acknowledgments:} This work originated during the author's visit to prof. Bo Berndtsson at Chalmers Technical University in Goteborg. The author is very grateful for the hospitality he was provided there, especially to prof. Bo Berndtsson for many illuminating discussions. The research was partially supported by the National Science Center of Poland grant no. 2019/35/N/ST1/01372.

\section{Preliminaries}

For a more elaborate discussion we refer to the Preliminary sections of \cite{Sr21, DS21}. We call $(M,I,J,K)$ a hypercomplex manifold provided the complex structures $I$, $J$ and $K$ satisfy
\[ \label{hypercomplex} \tag{2.1} IJK=-id_{TM}.\]
We remark that for us endomorphisms fields act from the right on $TM$ and from the left on differential forms. The latter action is given by
\[E \alpha = \alpha(\cdot E, ... , \cdot E)\]
for an endomorphisms field $E$ and differential form $\alpha$. A Riemannian metric $g$ on $(M,I,J,K)$ is called hyperhermitian if 
\[\label{hyperhermitian} \tag{2.2} g=Ig=Jg=Kg\]
and we call $(M,I,J,K,g)$ a hyperhermitian manifold. 

If not state otherwise once the hypercomplex structure is fixed we refer to the properties of complex differential forms, eg. being of type $(p,q)$, always with respect to $I$. Also the symbols $\pa$ and $\bpar$ are reserved for the Dolbeault operators of the induced complex manifold $(M,I)$. We define their twisted versions by 
\[ \begin{gathered} \jpar = J^{-1} \circ \bpar \circ J \\  \bjpar =  J^{-1} \circ \pa \circ J, \end{gathered} \]
those were introduced by Verbitsky, cf. \cite{AV10}. Notice that since $J$ is a real operator it is true that
\[\overline{\jpar \alpha} = \bjpar \overline{\alpha}\]
for any differential form $\alpha$. It is easy to prove that
\[\label{jaction} \tag{2.3} J \: : \: \Lambda^{p,q}M \longrightarrow \Lambda^{q,p}M.\]

As this is going to be heavily exploited by us, we discuss the infinitesimal model of the hyperhermitian manifold. Suppose we have a real vector space $V$ endowed with three almost complex endomorphisms  $I$, $J$, $K$ and an inner product $g$ satisfying (\ref{hypercomplex}) and (\ref{hyperhermitian}). We follow closely subsection 2.2 in \cite{DS21}, one can also find an even more elaborate discussion in \cite{H90}, Chapter 2. We can formally associate to $g$ a hyperhermitian sesquilinear form (in particular $\hh$ valued, where $\hh$ stands for the algebra of quaternions, so one has to be cautious about multi-linearity)
\[\tag{2.4} \label{decomp} H = g + \ii \omega_I + \jj \omega_J + \kk \omega_k.\]
Such $H$ is positive on non zero vector, $\hh$ linear on second entry and $\hh$ anti-linear in the first. It is an easy exercise to check that conversely every such $H$ is of the for (\ref{decomp}) for inner product $g$ satisfying (\ref{hyperhermitian}). 

We shuffle (\ref{decomp}) into
\[\label{shaff} \tag{2.5} H = (g + \ii \omega_I) + (\jj \omega_J + \kk \omega_k) := \bar{h} + \jj \Omega\]
where of course
\[\label{components} \tag{2.6} \begin{gathered} \bar{h} = g + \ii \omega_I, \\ \Omega = \omega_J - \ii \omega_K = \bar{h}(\cdot J,\cdot). \end{gathered}\]
We remark that this $\Omega$ is positive in a sens
\[\Omega(X,XJ) = \bar{h}(XJ,XJ) \geq 0\]
and we take that as a definition of being positive for a $J$-real $(2,0)$ form, see \cite{DS21} or \cite{Sr21} form more details. 

\begin{remark}
We use $\bar{h}$ in place of $h$ since the reader may be used to the convention that hermitian sesquilinear form is linear in first coordinate. One can view the form $\Omega$ - which we call the hyperhermitian form - as an analogue of the form $\omega$ in hermitian setting. As in the latter, $\Omega$ still determines a hyperhermitian metric completely. This is because from (\ref{components}) the form $\Omega$ controls $\bar{h}$ which in turn encodes $g$. We can repeat the above considerations without positivity assumption on $g$ which would just result in $\bar{h}$ or $\Omega$ not being positive definite.
\end{remark}

In particular, by choosing a basis for $V$ in which $I$, $J$, $K$ and $g$ are standard, we see that
\[\label{volpost} \tag{2.7} \Omega^n \wedge \overline{\Omega^n} = c(n) \omega_I^n = c(n) vol_g\]
for a dimensional constant $c(n)$.

\begin{remark} \label{coord}
Though we try to avoid local calculations for most of the time, it is always possible on hypercomplex manifold to choose special $I$ holomorphic coordinates in which, at the point, the hypercomplex structure and hyperhermitian metric are standard and moreover the first order derivatives of $J$ are zero, cf. \cite{DS21}. One can moreover assure that $\Omega+\pa \jpar \phi$ is diagonal, in the sens of (\ref{simul}), at the point in these coordinates. If not stated otherwise any local calculations below are always in that coordinates. 
\end{remark}

\section{Comparison between complex and quaternionic Hessians}

As announced in the introduction, we will prove in this section a version of an observation of Cheng and Yau on comparison, in their case, between real and complex Monge-Amp\`ere operators. We start with the comparison of the determinant of the complex Hessian and of, what we call, its hyperhermitian part for a plurisubharmonic function. 

\begin{lemma} %comparison of hermitian and its hyperhermitian part determinants
Let $\alpha$ be a non-negative real $(1,1)$ form on hypercomplex manifold $(M,I,J,K)$. The following inequality holds
\[ \tag{3.1} \label{complexquaternioniccomparison} \left( \alpha - \alpha(\cdot J,\cdot J) \right)^{2n} \geq \alpha^{2n}. \]
\end{lemma}
\begin{proof}
Let us denote
\[\beta:= - \alpha(\cdot J,\cdot J) .\]
Since $J$ is a real operator, $\beta$ is a real form. Because of the properties of the action of $J$ on $(p,q)$ forms, cf. (\ref{jaction}), $\beta$ is also a $(1,1)$ form. We claim that this is also non-negative. That is because
\[ \beta(Z,\overline{Z})=-\alpha(ZJ,\overline{Z}J)=\alpha(\overline{Z}J,ZJ)=\alpha(W,\overline{W}) \geq 0\]
since $\alpha$ is non-negative, $J$ is a real operator and $W$ is a $(1,0)$ vector.

We obtain
\[ \left( \alpha - \alpha ( \cdot J, \cdot J) \right)^{2n} = (\alpha + \beta)^{2n}= \binom{2n}{k} \alpha^k \wedge \beta^{2n-k} \geq \alpha^{2n} \]
since both $\alpha$ and $\beta$ are non-negative which is the claimed inequality (\ref{complexquaternioniccomparison}).
\end{proof}

Now we prove the comparison between complex and quaternionic Monge-Amp\`ere operators. This will be crucial in the proof of the main theorem in the next section. We will, using (\ref{3}), estimate the determinant of the linearization from below by the quaternionic Monge-Ampere operator and then, using the proposition below, estimate further from below by the complex Monge-Amp\`ere operator. This will constitute the core of the proof.

\begin{proposition} \label{comperr} %comparison of complex and quaternionic M.-A. operators
Let $\phi$ be a smooth function on a hypercomplex, not necessarily closed, manifold $(D,I,J,K)$ such that 
\[\ii \pa \bpar \phi \geq 0. \] 
The following comparison between $\phi$'s complex and quaternionic Monge-Amp\`ere operators holds
\[ \tag{3.2} \label{cmavsqma} (\pa \jpar \phi)^n \wedge \overline{(\pa \jpar \phi)^n} \geq c(n)(\ii \pa \bpar \phi)^{2n}, \]
for the dimensional constant $c(n)$.
\end{proposition}
\begin{proof}
One can easily calculate in local coordinates described in the Preliminaries, cf. \cite{DS21} for the calculation,
\[\label{quathess} \tag{3.3} \pa \jpar \phi = \phi_{i\bar{j}} dz_i \wedge J^{-1}d\overline{z_j} = \phi_{i \bar{j}} (dz_i \otimes J^{-1} d \overline{z_j} - J^{-1}d\overline{z_j} \otimes dz_i).\]
Also from the Preliminaries, cf. (\ref{volpost}) for $\Omega:= \pa \jpar \phi$ which is easily checked to be positive under the plurisubharmonicity assumption on $\phi$, we have that
\[ \label{volumetransition} \tag{3.4} (\pa \jpar \phi)^n \wedge \overline{(\pa \jpar \phi)^n} = c(n) \left((\pa \jpar \phi + \bpar \bjpar \phi)(\cdot I, \cdot J)\right)^{2n}.\]
We thus compute 
\[ \pa \jpar \phi + \bpar \bjpar \phi = \phi_{i \bar{j}} (dz_i \otimes J^{-1} d \overline{z_j} - J^{-1}d\overline{z_j} \otimes dz_i + d\overline{z_j} \otimes J^{-1} d z_i - J^{-1}dz_i \otimes d\overline{z_j}) \]
and consequently
\[ \label{hyperhermpart} \tag{3.5} \begin{gathered} 
(\pa \jpar \phi + \bpar \bjpar \phi)(\cdot I, \cdot J) = \\
\phi_{i \bar{j}} (I dz_i \otimes J \circ J^{-1} d \overline{z_j} - I \circ J^{-1}d\overline{z_j} \otimes J dz_i + I d\overline{z_j} \otimes J \circ J^{-1} d z_i - I \circ J^{-1}dz_i \otimes J d\overline{z_j}) = \\
\ii \phi_{i\bar{j}} (dz_i \otimes d \overline{z_j} - d\overline{z_j} \otimes d z_i) - \ii \phi_{i\bar{j}}(J dz_i \otimes J d\overline{z_j}- Jd\overline{z_j} \otimes J dz_i) = \\ \ii \pa \bpar \phi - J(\ii \pa \bpar \phi). \end{gathered}\]

Now the claim follows from (\ref{volumetransition}) and (\ref{complexquaternioniccomparison}) in the previous lemma with $\alpha := \ii \pa \bpar \phi$.
\end{proof}

\section{Sharp $C^0$ estimate for certain quaternionic PDEs on hyperhermitian manifolds}

The PDEs, generalizing the quaternionic Monge-Amp\`ere equation, we are going to consider in this section are those which factor through the eigenvalues of the hyperhermitian form $\Omega+ \pa \jpar \phi$ for a smooth function $\phi$. We do not intend to elaborate on that too much but, $\pa \jpar \phi$ can be viewed, what is especially transparent in the case of the flat space - cf. \cite{Sr18, Sr21}, as the counterpart of the real Hessian $\nabla^2 \phi$ and the complex one $\ii \pa \bpar \phi$. Moreover, in the flat case it it truly given by a matrix of certain second order derivatives, cf. \cite{Sr21}, the so called Fueter operators. 

Precisely, having a hyperhermitian manifold $(M,I,J,K,g)$ we denote the perturbed hyperhermitian form by
\[\Omega_\phi := \Omega + \pa \jpar \phi\]
for a smooth function $\phi$. From Preliminaries we know $\Omega_\phi$ defines an associated symmetric, quaternion invariant form in $TM$ which we denote by $g_\phi$. 

\begin{remark}
Complex geometers should be warned that, by far, the associated, with respect to $I$, hermitian form of $g_\phi$ \textbf{is not} obtained by perturbing $\omega_I$ by a complex Hessian, i.e. it \textbf{is not} $\omega_I+ \ii \pa \bpar \phi$. In fact $\omega_{I,\phi}$ is obtained by perturbing $\omega_I$ by a hyperhermitian part of the complex Hessian $\ii \pa \bpar \phi$ as we have seen in (\ref{hyperhermpart}). 
\end{remark}

We define the $n$-tuple of the eigenvalues, 
\[ \lambda(\Omega_\phi)=\lambda(g_\phi)=\lambda(\phi),\] to be obtained by taking the eigenvalues of the hyperhermitian endomorphism
\[g^{-1} \circ g_\phi: (TM,I,J,K) \longrightarrow T^{*}M \longrightarrow (TM,I,J,K),\]
with their multiplicities divided by four. It is trivial to verify the eigenvalues of $g^{-1} \circ g_\phi$ have multiplicities divisible by four since this endomorphism is hermitian with respect to $I$, $J$ and $K$. We will often use an equivalent definition of $\lambda(\phi)$. Namely, $\lambda(\phi)$ is formed by the numbers $\lambda_i$ such that there is a frame $\{e_i\}_{i=0}^{2n-1}$ of $T^{1,0}M$ in which
\[ \label{simul} \tag{4.1} \begin{gathered} e_{2i+1} = e_{2i}J, \\ \Omega = e_{2i}^* \wedge e_{2i+1}^*, \\ \Omega_\phi = \lambda_i e_{2i}^* \wedge e_{2i+1}^*. \end{gathered}\] 

In Proposition \ref{qpdebp} we treat the equations of the form: 
\[ f \left(\lambda(\phi)\right) = e^F \]
for certain $f$'s which we describe now. We assume that:
\[ \label{1} \tag{4.2} 
\begin{gathered} f: \rr^n \supset \Gamma_f \longrightarrow (0,+\infty), \\ 
 f(1,...,1)=1, \\
 \{\lambda \in \rr^n \: | \: \lambda_i > 0\} \subset \Gamma \subset \{ \lambda \in \rr^n \: | \: \lambda_1 + ... + \lambda_n >0\}, \end{gathered} \]
where $f$ is symmetric in its arguments, homogeneous of degree one and $\Gamma$ is a symmetric cone,
\[\label{2} \tag{4.3} \frac{\pa f}{\pa \lambda_i}(\lambda) >0 \text{ for all } 1 \leq i \leq n \text{ and any } \lambda \in \Gamma,\]
\[\label{3} \tag{4.4} \prod_{i=1}^n\frac{\pa f}{\pa \lambda_i}(\lambda) \geq \gamma \text{ for some } \gamma>0 \text{ and any } \lambda \in \Gamma.\]
These are exactly the assumptions on $f$ from \cite{GP22}. Conditions (\ref{1}) and (\ref{3}) are important in the proof of Proposition \ref{qpdebp}. Especially the latter one which tells us that the determinant of the coefficients of the linearized operator is bounded from below. The first one gives also the normalization of the way in which PDE is written. Condition (\ref{2}) secures ellipticity. 

From \cite{HL22} we know that, in particular, any homogenized to degree one homogeneous symmetric polynomial which, after homogenization, dominates the $n$'th root of determinant satisfies (\ref{3}).

We are ready to state and proof the result from which Theorem \ref{mt} will follow.

\begin{proposition} \label{qpdebp}
Let $(M,I,J,K,g)$ be a closed, connected hyperhermitian manifold and let $\phi$ be a smooth solution of the following PDE 
\[\tag{4.5} \label{qpde} \begin{cases} f \left(\lambda(\phi) \right) = e^F \\ \lambda(\phi) \in \Gamma_f \\ sup_M \phi = 0 \end{cases} \] where $f$ and $\Gamma_f$ satisfy (\ref{1})-(\ref{3}). For any $p>2n$, where $n$ is the quaternionic dimension of $M$, there is a uniform bound
\[\label{qpdeb} \tag{4.6} - \inf_M \phi < C \] where the constant $C$ depends only on $p$, geometry of $(M,I,J,K,g)$ and $\parallel e^{2nF} \parallel_{L^1(\log L)^p}$.
\end{proposition}
\begin{proof}[Proof of Proposition \ref{qpdebp}]

The following setting is similar to the one in the proof of Theorem 1 in \cite{GP22}. We recall it only for the readers convenience.  

One can cover the manifold under investigation with a collection of $I$ holomorphic coordinate balls $B(z,2r)$, still contained in twice as big coordinate balls, with radii $r\leq \frac 1 2$ and centers in all points $z$ of the manifold. Moreover, one can assume that in each of these balls the holomorphic coordinates are actually chosen so that the properties from Remark \ref{coord} are satisfies and additionally in any ball $B(z,2r)$
\[\label{coordbound} \tag{4.7} \frac 1 2 \pa \jpar |z|^2 \leq \Omega. \]
Equivalently, if the reader is insecure about this formalism, the hyperhermitian metric $g$ is bounded up to the constants from below by the flat metric in the chosen coordinates. Let $z_0$ be the point where $\phi$ is achieving its absolute minimum. We can assume 
\[ -\phi(z_0) \geq 2\] 
since otherwise we do have an uniform bound. 

Let us consider, after Guo and Phong - cf. (2.5) in \cite{GP22}, the functions
\[\label{usdef} \tag{4.8} u_s(z) = \phi(z) - \phi(z_0) + \frac 1 2 |z|^2 - s\]
for any $s \in (0,2r^2)$ in the ball $B(z_0,2r)$. 
It is easy to see that $u_s>0$ on $\partial \ball$. We denote the associated sub-level sets of $u_s$ by
\[B_s = \{z \in \ball \: | \: u_s(z) <0 \},\]
their closures are clearly compact in $\ball$.

Now we are about to set the auxiliary \textit{complex Monge-Amp\`ere equation} with whose solutions we are going to compare $u_s$. Namely, set the Dirichlet problem, motivated by $(2.9)$ in \cite{GP22}, 
\[ \tag{4.9} \label{cmadir} \begin{cases}
(\ii \pa \bpar \psi_{s,k} )^{2n} = \frac{\tau_k(-u_s)}{A_{s,k}}e^{2nF}\omega^{2n}_I \text{ in } \ball \\
\ii \pa \bpar \psi_{s,k} \geq 0 \\
\psi_{s,k} = 0, \text{ on } \partial \ball
\end{cases} \] 
where $\tau_k$ are smooth strictly positive functions decreasing to $x \cdot \chi_{\mathbb{R}_+}(x)$ on $\mathbb{R}$ and
\[A_{s,k} = \int_{\ball} \tau_k(-u_s) e^{2nF}\omega^{2n}_I.\] 
The smooth solutions to (\ref{cmadir}) exist by the classical result of Caffarelli et al. \cite{CKNS85}. 

Note that due to the assumptions on $\tau_k$
\[ \tag{4.10} \label{Aconverg} A_{s,k} \longrightarrow \int_{B_s} (-u_s)e^{2nF}\omega^{2n}_I:=A_s\]
and for later reference we remark that 
\[\tag{4.11} \label{psinormalization} \int_{\ball} (\ii \pa \bpar \psi_{s,k})^{2n} = 1. \]

\textbf{We claim that:}  
\[\tag{4.12} \label{usbound} -u_s \leq \epsilon (-\psi_{s,k})^{\frac{2n}{2n+1}}\]
in $\ball$ where 
\[\tag{4.13} \label{epsilonform} \epsilon^{2n+1} = C(n,\gamma) A_{s,k}.\]

\begin{proof}[Proof of the \textbf{Claim}]

To prove this claim let us consider the function
\[\Psi = -u_s - \epsilon (-\psi_{s,k})^{\frac{2n}{2n+1}}\]
which after rephrasing (\ref{usbound}) is needed to be non positive in $\ball$. 

From the definition of $u_s$, cf. (\ref{usdef}), and the fact that $\psi_{s,k}$ is non-positive - $\Psi$'s positive maximum can occur only inside $B_s$, at some $z_{max}$. By the maximum principle the linearization $L_f(\phi)$, at $\phi$ of the operator $\log f$ for $f$ from (\ref{qpde}), of $\Psi$ at $z_{max}$ is non positive.

Thus we write
\[\label{org} \tag{4.14} 0 \geq L_f(\phi ).\Phi = L_f(\phi ).(-u_s) -\epsilon L_f(\phi).(-\psi_{s,k})^{\frac{2n}{2n+1}}.\]
Observe that
\[ \label{precos} \tag{4.15} \pa \jpar (-u_s) = - \pa \jpar \phi - \frac{1 }{2} \pa \jpar |z|^2 \geq - \pa \jpar \phi - \Omega , \]
where in the last inequality we used the assumption (\ref{coordbound}). Due to ellipticity this results in
\[ \label{cos} \tag{4.16} L_f(\phi).(-u_s) \geq - L_f(\phi).\Omega_\phi. \]
Note that since the operator is linear the last expression makes sense so we use this abused notation through the whole proof. 

The last inequality actually means that
\[ \label{hmm} \tag{4.17} L_f(\phi).(-u_s) \geq - C, \] 
due to homogeneity assumption on the operator $f$ (Euler formula gives that the last term in (\ref{cos}) is a constant).

Let us now turn our attention to the remaining term in (\ref{org}). For that goal we compute
\[ \begin{gathered} 
\pa \jpar \left( (-\psi_{s,k})^{\frac{2n}{2n+1}} \right) = \\
\frac{2n}{2n+1} \pa \left( (-\psi_{s,k})^{\frac{-1}{2n+1}} \jpar (-\psi_{s,k}) \right) = \\
-\frac{2n}{(2n+1)^2}(-\psi_{s,k})^{\frac{-2n-2}{2n+1}} \pa (-\psi_{s,k})\wedge \jpar (-\psi_{s,k}) + \frac{2n}{2n+1} (-\psi_{s,k})^{\frac{-1}{2n+1}} \pa \jpar (-\psi_{s,k}) \leq \\
\frac{2n}{2n+1} (-\psi_{s,k})^{\frac{-1}{2n+1}} \pa \jpar (-\psi_{s,k}), \end{gathered}\]
where we have used the fact that the form 
\[ \alpha \wedge J^{-1} \overline{\alpha}\] 
is positive for any $(1,0)$ form $\alpha$. 

Arguing as above, in (\ref{precos}) and (\ref{cos}), we obtain 
\[\label{another} \tag{4.18} \begin{gathered} 
-\epsilon L_f(\phi).(-\psi_{s,k})^{\frac{2n}{2n+1}} = \\
-\epsilon L_f(\phi).\left( \pa \jpar (-\psi_{s,k})^{\frac{2n}{2n+1}} \right) \geq \\
\epsilon \frac{2n}{2n+1} (-\psi_{s,k})^{\frac{-1}{2n+1}} L_f(\phi).\pa \jpar (\psi_{s,k})= \\
\epsilon \frac{2n}{2n+1} (-\psi_{s,k})^{\frac{-1}{2n+1}} L_f(\phi).\psi_{s,k}. \end{gathered}\]

Putting (\ref{hmm}) and (\ref{another}) in (\ref{org}) results in
\[ \label{como} \tag{4.19} \begin{gathered} 
0 \geq -C + \epsilon \frac{2n}{2n+1} (-\psi_{s,k})^{\frac{-1}{2n+1}} L_f(\phi).\psi_{s,k} = \\
-C + \epsilon \frac{2n}{2n+1} (-\psi_{s,k})^{\frac{-1}{2n+1}} \left( \sum_{i=1}^{n}\frac{\frac{\partial f}{\partial \lambda_i}\left(\lambda(\phi)\right)}{f\left(\lambda(\phi)\right)} (\psi_{s,k;2i\overline{2i}} + \psi_{s,k;2i+1 \overline{2i+1}}) \right) \geq \\
-C + \epsilon \frac{2n^2}{2n+1} (-\psi_{s,k})^{\frac{-1}{2n+1}} \left( \prod_{i=1}^{n}\frac{\frac{\partial f}{\partial \lambda_i}\left(\lambda(\phi)\right)}{f\left(\lambda(\phi)\right)}\right)^{\frac{1}{n}} \left( \frac{(\pa \jpar \psi_{s,k})^n}{\Omega^n} \right)^{\frac 1 n} \geq \\
-C + \epsilon C' \gamma \frac{2n^2}{2n+1} (-\psi_{s,k})^{\frac{-1}{2n+1}} \frac{1}{f\left(\lambda(\phi)\right)} \left( \frac{(\ii \pa \bpar \psi_{s,k})^{2n}}{\omega^{2n}_I} \right)^{\frac{1}{2n}} =\\
-C + \epsilon C' \gamma \frac{2n^2}{2n+1} (-\psi_{s,k})^{\frac{-1}{2n+1}} \frac{1}{e^F} \frac{(\tau_k(-u_s))^{\frac{1}{2n}}}{A_{s,k}^{\frac{1}{2n}}}e^{F} \geq \\
-C + \epsilon C' \gamma \frac{2n^2}{2n+1} (-\psi_{s,k})^{\frac{-1}{2n+1}} \frac{(-u_s)^{\frac{1}{2n}}}{A_{s,k}^{\frac{1}{2n}}}.
 \end{gathered}\]
In the above: by no harm we assumed that $\Omega_\phi$ is diagonal in the first equality, take into an account the formula (\ref{quathess}) for the quaternionic Hessian $\pa \jpar \psi_{s,k}$ in coordinates from Remark \ref{coord}; second inequality is just the arithmetic geometric mean inequality; third inequality follows from Proposition \ref{comperr}; the last inequality follows from properties of $\tau_k$. 
 
After rewriting we conclude from (\ref{como}) that
\[\frac{C(n,\gamma)}{\epsilon^{2n}} (-\psi_{s,k})^{\frac{2n}{2n+1}} A_{s,k}  \geq -u_s.\] 
In order to obtain (\ref{qpdeb}) we can just set $\epsilon$ to satisfy
\[\frac{C(n,\gamma)}{\epsilon^{2n}} A_{s,k} = \epsilon. \]
\end{proof}

Coming back to the proof of Proposition \ref{qpdebp}, we note that due to (\ref{usbound}), the form of $\epsilon$ in (\ref{epsilonform}), the properties (\ref{psinormalization}) and (\ref{cmadir}) the functions $u_s$ and $\psi_{s,k}$ satisfy all the assumptions of Lemma $2$ in \cite{GP22}. Consequently, we obtain from that lemma the requested bound (\ref{qpdeb}). Let us just note that, though it is omitted in $(2.15)$ of \cite{GP22}, the constant there of course depends on the family $\{A_s\}_{s \in (0,2r^2)}$ or equivalently due to their Lemma $4$ on the entropy. 

Let us also note that in our setting we do have a bound on the $L^1$ norm of $\phi$, as originally derived for the quaternionic Monge-Amp\`ere equation by Alesker and Shelukhin in \cite{AS13}. 

This is because from the assumption (\ref{1})
\[\frac{(\Omega + \pa \jpar \phi)\wedge \Omega^{n-1}}{\Omega^n} \geq 0\]
this results in
\[\frac{\pa \jpar \phi \wedge \Omega^{n-1}}{\Omega^n} \geq -C\]
and now the classical argument for the existence of a bounded Green's function, like in \cite{AS13}, for this linear operator produces the desired $L^1$ bound. 
\end{proof}

Using the above proposition it is trivial to justify Theorem \ref{mt} what we now do.

\begin{proof}[Proof of Theorem \ref{mt}]

Define the operator 
\[ f(\phi) = \left(\frac{(\Omega+ \pa \jpar \phi )^n}{\Omega^n}\right)^{\frac{1}{n}}.\]
Clearly we are in a position to apply Proposition \ref{qpdebp}. It will provide us with the bound
\[- \inf_M \phi < C \]
where the dependence on the right hand side is controlled by 
\[ \parallel \left( e^{\frac F n} \right)^{2n}\parallel_{L^1(\log L)^p} = \parallel e^{2F} \parallel_{L^1(\log L)^p}, \]
as desired.

This quantity, in particular, controls $L^q$ norm of $e^F$ for any $q>2$. 
\end{proof}

\begin{remark}
One could wonder whether instead of enriching the described method by Cheng-Yau type trick one could just repeat the proof from \cite{GP22} with the complex Monge-Amp\`ere equation exchanged for the quaternionic one. There are at least two problems with this approach. First of all, the class of admissible functions for the quaternionic Monge-Amp\`ere equation is much wilder then the class of plurisubharmonic functions. Concretely, the analogue of $(2.16)$ in \cite{GP22} fails for those. It was shown in \cite{Sr18} that those function are locally integrable with any exponent $p<2$. Potentially this cold be enough but still we do not know if this bound is uniform in terms of the quaternionic Monge-Amp\`ere mass and, as mentioned in the Introduction, this flat picture does not represent the local situation on generic hypercomplex manifold. More problematic is the analogue of \cite{CKNS85}, i.e. the solvability of the smooth Dirichlet problem locally. Again, contrary to the complex setting, where the equation always looks the same locally, there is no one local model of the quaternionic Monge-Amp\`ere equation we are forced to solve. Even if we would be able to do that, this would required a substantial and elaborate additional work in comparison to the method we present.

Let us also note that the applicability of the reasoning from \cite{GPT21} is even more questionable in the proof of Theorem \ref{mt}. This is partially from the reason already noted in \cite{GP22}. For equation (\ref{qma}) there is no a priori normalization on the right hand side guaranteeing the solvability, i.e. every instantiation of quaternionic Monge-Amp\`ere equation is already a torsion one. More importantly, the issue is that in our case the mentioned conjecture from \cite{AV10} is still unsolved. As a result - if we choose as an auxiliary equation the global quaternionic Monge-Amp\`ere equation, we do not know whether it is actually solvable.  
\end{remark}

\noindent {CIRGET LABORATORY \\ 
CENTRE DE RECHERCHES MATH\'EMATIQUES (CRM) \\
PRESIDENT KENNEDY BUILDING, 201 AVENUE DU PR\'ESIDENT KENNEDY \\
MONTR\'EAL, QU\'EBEC, H2X 3Y7 \\
CANADA \\
\textit{E-mail address:} sroka.marcin@uqam.ca}

\bigskip

\noindent {FACULTY OF MATHEMATICS AND COMPUTER SCIENCE\\
OF JAGIELLONIAN UNIVERSITY \\
\L OJASIEWICZA  6 \\
30-348, KRAK\'OW \\
POLAND \\
\textit{E-mail address:} marcin.sroka@im.uj.edu.pl}

\end{document}